\documentclass[reqno,twoside,12pt]{amsart}

\hoffset - 1.5cm

\usepackage{amsmath}
\usepackage{amssymb}
\usepackage{graphicx}
\usepackage{amstext}
\usepackage{amsopn}
\usepackage{amsfonts}
\usepackage{xcolor}
\usepackage{setspace}
\usepackage{enumerate}
\usepackage[polish,english]{babel}
\usepackage[utf8]{inputenc}
\usepackage{polski}

\usepackage{lmodern}

\newtheorem{Remark}{Remark}[section]

\newtheorem{Lemma}[Remark]{Lemma}
\newtheorem{Proposition}[Remark]{Proposition}
\newtheorem{Theorem}[Remark]{Theorem}

\newcommand{\bC}{\mathbb{C}}

\newcommand{\bH}{\mathbb{H}}

\newcommand{\bR}{\mathbb{R}}

\newcommand{\bV}{\mathbb{V}}

\newcommand{\cC}{\mathcal{C}}

\newcommand{\cI}{\mathcal{I}}

\newcommand{\cN}{\mathcal{N}}

\newcommand{\cT}{\mathcal{T}}

\newcommand{\fg}{\mathfrak{g}}
\newcommand{\fh}{\mathfrak{h}}

\newcommand{\fl}{\mathfrak{l}}

\newcommand{\fp}{\mathfrak{p}}

\newcommand{\sub}{\overline{\mathrm{sub}}}

\newcommand{\Id}{\mathrm{Id}}

\makeatletter
\@namedef{subjclassname@2020}{%
  \textup{2020} Mathematics Subject Classification}
\makeatother

\usepackage{todonotes}
\usepackage{setspace}

\begin{document}

\title[Unbounded continua for Yamabe-type equations]{Geometrically distinct unbounded continua of solutions for Yamabe-type equations}

\author{Piotr Stefaniak}
\address{Faculty of Mathematics and Computer Science \\ Nicolaus Copernicus University in Toru\'n\\
PL-87-100 Toru\'{n} \\ ul. Chopina $12 \slash 18$ \\
Poland}

\email{cstefan@mat.umk.pl (P. Stefaniak)}

\date{\today}


\begin{abstract}
We study global bifurcation for Yamabe-type equations on closed Riemannian manifolds with symmetries. For isometric group actions with interval orbit space and a suitable range of exponents, we prove the existence of pairwise disjoint unbounded continua bifurcating from the positive eigenvalues of the Laplace--Beltrami operator that admit invariant eigenfunctions. All nontrivial solutions in these continua change sign. The results cover critical and some supercritical exponents. On spheres and complex and quaternionic projective spaces, we obtain lower bounds on the number of geometrically distinct continua bifurcating from the same eigenvalue.
\end{abstract}

\maketitle

\section{Introduction}\label{sec:introduction}

Let \(M\) be a closed connected \(n\)-dimensional Riemannian manifold. We study the global structure of the weak solutions of the Yamabe-type equation
\begin{equation}\label{eq:yamabe-int}
-\Delta u-\lambda u=|u|^{p-2}u,
\end{equation}
where \(-\Delta\) is the Laplace--Beltrami operator on \(M\), \(p>2\), and \(\lambda\in\bR\). When \(n\geq3\) and the background metric has constant scalar curvature, this family includes, at the critical exponent \(p=2n/(n-2)\), the equation for conformal metrics of positive constant scalar curvature; see \cite{Aubin1982,LeeParker1987}.

Our main results give sufficient conditions for the existence of unbounded continua of sign-changing solutions, each meeting the trivial branch only at its bifurcation point. Continua bifurcating from different eigenvalues are therefore disjoint. We also use the symmetries of \(M\) to distinguish continua bifurcating from the same eigenvalue whose solutions are invariant under different groups of isometries.

Symmetry assumptions also play a central role in several related works on Yamabe-type equations. Clapp and Fern\'andez obtained multiplicity results for sign-changing solutions of critical Yamabe-type equations on compact manifolds in \cite{ClappFernandez2017}. Isoparametric functions and cohomogeneity one actions have also been used to construct such solutions. Henry obtained nodal (i.e., sign-changing) solutions constant along the level sets of an isoparametric function in \cite{Henry2019}. In \cite{FerPalPet2020}, Fern\'andez, Palmas and Petean obtained infinitely many nodal solutions of Yamabe-type equations on spheres for a range of exponents including supercritical ones, with arbitrarily many connected components of the nodal set and blow-up along focal submanifolds. They also constructed sequences of nodal solutions of the critical Yamabe equation on complex and quaternionic projective spaces with unbounded maximum norms. In \cite{ClappSaldanaSzulkin2021}, Clapp, Salda\~na and Szulkin used optimal partitions to construct solutions of the Yamabe equation on the sphere with a prescribed number of nodal domains, minimizing the energy among solutions invariant under the same block action and having that number of nodal domains.

These results concern existence and multiplicity for fixed equations. Bifurcation under variations of the background metric was studied by de Lima, Piccione and Zedda on product manifolds in \cite{LimaPiccioneZedda2012}, and by Bettiol and Piccione on spheres in \cite{BettiolPiccione2013}. These works establish bifurcation and local rigidity for the Yamabe problem using spectral properties of the Laplacian, changes in the Morse index, and equivariant bifurcation methods. For fixed background metrics, Petean and Barrantes Gonz\'alez used global bifurcation in \cite{PetBar2019} to obtain multiplicity of positive solutions to subcritical equations on products of spheres, invariant under the diagonal action. The same paper also establishes the existence of infinitely many invariant nodal solutions of the critical Yamabe equation on these products. Under more general isoparametric assumptions, Betancourt de la Parra, Julio-Batalla and Petean studied positive solutions bifurcating from a branch of constant solutions in \cite{BetJulPet2021}. Their global bifurcation argument separates branches by counting intersections with the constant solution and gives multiplicity of positive solutions for sufficiently large parameter values.

Here we study bifurcation from the zero branch and compare continua of solutions invariant under the same or different subgroups of a group of isometries. More precisely, let a compact Lie group \(G\) act isometrically on \(M\), and let \(K\subset G\) be a closed subgroup such that \(M/K\) is a closed interval.
When the space \((H^1(M))^K\) of \(K\)-invariant functions in \(H^1(M)\) is compactly embedded in \(L^p(M)\), Theorem~\ref{thm:unbounded-continuum} gives an unbounded continuum in \((H^1(M))^K\times\bR\) bifurcating from every \(K\)-invariant eigenvalue of \(-\Delta\). Each continuum meets the trivial branch only at its bifurcation point, and those arising from different eigenvalues in the same invariant space are disjoint. Their nontrivial solutions change sign when the bifurcation eigenvalue is positive and have one sign when it is zero. For \(n\geq3\), the improved compactness under symmetries provided by \cite{HebVau} allows the result to cover critical and some supercritical exponents whenever all \(K\)-orbits have positive dimension.

The separation argument regards a nontrivial solution \(u\) as an eigenfunction of \(A_u=-\Delta-|u|^{p-2}\), with eigenvalue \(\lambda\). The ordered \(K\)-invariant eigenvalues vary continuously with \(u\) by Lemma~\ref{lem:eigenvalue-continuity}. When they are simple, the position of \(\lambda\) in their ordering is constant on each connected component by Lemma~\ref{lem:index-continuity}. This prevents a component from joining different bifurcation points, as stated in Theorem~\ref{thm:spectral-index}. For actions with interval orbit space, we prove simplicity in Theorem~\ref{thm:schrodinger-simplicity} using the reduction to an ordinary differential equation established in Lemma~\ref{lem:invariant-eigenfunction-reduction}. The separation of bifurcation points then excludes the bounded alternative in the Rabinowitz global bifurcation theorem from \cite{Rabinowitz1971}, so the components are unbounded. Finally, Proposition~\ref{prop:sign-changing} shows that their nontrivial solutions change sign when they bifurcate from a positive eigenvalue and have one sign when they bifurcate from zero.

Symmetry and preservation of nodal properties were used by Healey and Kielh\"ofer in \cite{HealeyKielhofer1993} to prove separation and unboundedness of global branches of elliptic systems. A corresponding argument based on Sturm theory would require control of the endpoint behaviour of the reduced equation to ensure that the number of zeros remains constant along connected components. Our proof uses regularity only in the interior of the orbit interval and the integral identity defining eigenfunctions of \(A_u\) on \(M\) to establish simplicity, without requiring regularity at the nonprincipal orbits.

For homogeneous spaces \(M=G/H\), we give group-theoretic conditions under which Theorem~\ref{thm:unbounded-continuum} applies. Let \(G\) be compact, connected and semisimple, and let \(H,K\in\sub(G)\) be connected. Lemma~\ref{lem:positive-dimensional-orbits} shows that all \(K\)-orbits have positive dimension precisely when no conjugate of \(K\) is contained in \(H\). If the action of \(K\) on \(G/H\) has cohomogeneity one, Lemma~\ref{lem:orbit-spaces-interval} shows that its orbit space is a closed interval. Together, these conditions allow the bifurcation theorem to be applied at the critical exponent and for some supercritical exponents when \(\dim M\geq3\).

We then distinguish continua bifurcating from the same eigenvalue whose solutions are invariant under different subgroups of \(G\). Two solutions at the same parameter are \(G\)-geometrically equivalent if one is obtained from the other by the action of an element of \(G\). When \(G\) acts transitively on \(M\), Proposition~\ref{prop:geometric-separation} gives a condition on intersections of fixed-point spaces that excludes this equivalence between nontrivial solutions in different continua of sign-changing solutions. The argument follows the fixed-point-space method used in \cite{Marzantowicz2019}.

We apply these results to the critical equation on \(S^{m-1}\), \(\bC P^{m-1}\), and \(\bH P^{m-1}\), endowed with their standard metrics, where \(m\geq4\). We consider the block diagonal actions of \(SO(a)\times SO(m-a)\), \(S(U(a)\times U(m-a))\), and \(Sp(a)\times Sp(m-a)\), respectively, for \(2\leq a\leq\lfloor m/2\rfloor\). Theorem~\ref{thm:multiplicity-block-actions} gives at least \(\lfloor m/2\rfloor-1\) unbounded continua bifurcating from each positive eigenvalue common to these actions. Their nontrivial solutions change sign, and their nontrivial parts are pairwise geometrically distinct under \(SO(m)\), \(SU(m)\), and \(Sp(m)\), respectively. On the projective spaces, additional symmetry groups increase the bound to \(\lfloor m/2\rfloor\) at every second positive eigenvalue in the common sequence. These bounds count continua issuing from a common bifurcation point, with geometric separation along their entire nontrivial parts.

The paper is organized as follows. Section~\ref{sec:connected-components} develops the separation argument and proves simplicity and the assertions about signs for actions with interval orbit space. Section~\ref{sec:selman} establishes global bifurcation and verifies the assumptions for actions on homogeneous spaces. It then gives the geometric separation criterion and the multiplicity results on spheres and projective spaces. The appendix (Section~\ref{appendix}) contains a proof of the maximality property of symmetric subgroups and computations of the smallest orbit dimensions for the actions considered in the paper.

\section{Connected components of invariant solutions}\label{sec:connected-components}

Let \(M\) be a closed connected Riemannian manifold and let \(G\) be a compact Lie group acting isometrically on \(M\). We consider the Yamabe-type equation
\begin{equation}\label{eq:main}
-\Delta u-\lambda u=|u|^{p-2}u,
\end{equation}
where \(p>2\), \(\lambda\in\bR\), and \(-\Delta\) denotes the Laplace--Beltrami operator on \(M\).

Our aim in this section is to give sufficient conditions for the unboundedness of continua of nontrivial solutions and for all nonconstant solutions in these continua to change sign.

\subsection{Separation of connected components}\label{subsec:separation-components}

Let \(H^1(M)\) be the first Sobolev space on \(M\). Since the action of \(G\) on \(M\) is isometric, \(H^1(M)\), equipped with its standard inner product and the action \((gu)(x)=u(g^{-1}x)\), is an orthogonal \(G\)-representation. For a closed subgroup \(K\subset G\), let
\[
(H^1(M))^K=\{u\in H^1(M):ku=u\text{ for every }k\in K\}
\]
be its fixed-point space. We assume that \((H^1(M))^K\) is continuously embedded in \(L^p(M)\).

A pair \((u,\lambda)\in(H^1(M))^K\times\bR\) is called a \(K\)-invariant weak solution of \eqref{eq:main} if
\begin{equation*}
\int_M\langle\nabla u,\nabla v\rangle dM-\lambda\int_Muv dM=\int_M|u|^{p-2}uv dM
\end{equation*}
for every \(v\in(H^1(M))^K\).

For every \(\lambda\in\bR\), the pair \((0,\lambda)\) is a weak solution of \eqref{eq:main}. These solutions form the trivial branch
\(
\cT_K=\{(0,\lambda):\lambda\in\bR\}\subset(H^1(M))^K\times\bR.
\)
We denote the set of nontrivial \(K\)-invariant weak solutions by
\begin{equation*}
\cN_K=\left\{(u,\lambda)\in(H^1(M))^K\times\bR:u\neq0,\ (u,\lambda)\text{ is a weak solution of \eqref{eq:main}}\right\}.
\end{equation*}
A point \((0,\lambda)\in\cT_K\) is called a bifurcation point if \((0,\lambda)\in\overline{\cN_K}\).

We use the following spectral notation. For any \(K\)-equivariant operator \(T\), i.e., satisfying \(T(kv)=kT(v)\) for every \(k\in K\) and every \(v\) in its domain, we call the spectrum of its restriction to the \(K\)-fixed subspace the \(K\)-invariant spectrum of \(T\) and denote it by \(\sigma_K(T)\). An eigenvalue \(\mu\) of this restriction is called a \(K\)-invariant eigenvalue of \(T\). We denote its eigenspace by \(\bV_T^K(\mu)\) and call \(\dim\bV_T^K(\mu)\) the \(K\)-invariant multiplicity of \(\mu\).

We begin with a necessary condition for bifurcation, which is a standard consequence of the implicit function theorem.

\begin{Lemma}\label{lem:index-trivial}
If \((0,\lambda)\in\overline{\cN_K}\), then \(\lambda\in\sigma_K(-\Delta)\).
\end{Lemma}

To study the connected components of \(\overline{\cN_K}\), we associate with every \(u\in(H^1(M))^K\) the linear operator \(A_u:(H^1(M))^K\longrightarrow((H^1(M))^K)^*\), defined by
\[
\langle A_uv,w\rangle=\int_M\langle\nabla v,\nabla w\rangle dM-\int_M|u|^{p-2}vw dM
\]
for \(v,w\in(H^1(M))^K\), where \(\langle A_uv,w\rangle\) denotes the value of the functional \(A_uv\) at \(w\). Formally, \(A_u=-\Delta-|u|^{p-2}\). If \((u,\lambda)\in\cN_K\), then \(A_uu=\lambda u\) in the weak sense, so \(\lambda\) is a \(K\)-invariant eigenvalue of \(A_u\) with eigenfunction \(u\).

This observation is the starting point for the main result of this section.

\begin{Theorem}\label{thm:spectral-index}
Assume that \(A_u\) has simple \(K\)-invariant spectrum for every \(u\in(H^1(M))^K\). Then every connected component of \(\overline{\cN_K}\) intersects \(\cT_K\) in at most one point.
\end{Theorem}

We postpone the proof until we have established several auxiliary lemmas concerning the ordered eigenvalues of \(A_u\) and their continuity with respect to \(u\). These will allow us to define an integer-valued index on \(\overline{\cN_K}\) that is constant on connected components.

To study these eigenvalues, we will construct a compact self-adjoint operator to which the spectral theorem applies. Since \(A_u\) need not be invertible, we first add a sufficiently large constant \(c\). We will show that, for every \(f\in(L^2(M))^K\), the equation
\begin{equation}\label{eq:shifted-schrodinger}
(A_u+c)v=f
\end{equation}
has a unique weak solution \(v\in(H^1(M))^K\). The map assigning this solution to \(f\) will be denoted by \((A_u+c)^{-1}\). Viewed as an operator on \((L^2(M))^K\), it will be compact and self-adjoint, and its eigenvalues will determine those of \(A_u\).

For \(u\in(H^1(M))^K\) and \(c\geq0\), define the bilinear form \(B_u:(H^1(M))^K\times(H^1(M))^K\to\bR\) by
\[
B_u(v,w)=\langle(A_u+c)v,w\rangle=\int_M\langle\nabla v,\nabla w\rangle dM-\int_M|u|^{p-2}vw dM+c\int_Mvw dM.
\]
Thus \(v\in(H^1(M))^K\) is a weak solution of \eqref{eq:shifted-schrodinger} if
\(
B_u(v,w)=\int_Mfw dM
\)
for every \(w\in(H^1(M))^K\).

\begin{Lemma}\label{lem:shifted-schrodinger}
For every \(u\in(H^1(M))^K\) and all sufficiently large \(c>0\), equation \eqref{eq:shifted-schrodinger} has a unique weak solution \(v\in(H^1(M))^K\) for every \(f\in(L^2(M))^K\). Moreover, the solution satisfies \(\|v\|_{H^1(M)}\leq2\|f\|_{L^2(M)}\).
\end{Lemma}

\begin{proof}
Fix \(u\in(H^1(M))^K\). Splitting the integral \(\int_M|u|^{p-2}v^2 dM\) over the sets where \(|u|^{p-2}\leq R\) and \(|u|^{p-2}>R\), and applying H\"older's inequality and the continuous embedding of \((H^1(M))^K\) into \(L^p(M)\), we obtain
\[
\int_M|u|^{p-2}v^2 dM\leq R\|v\|_{L^2(M)}^2+C\left(\int_{\{|u|^{p-2}>R\}}|u|^p dM\right)^{\frac{p-2}{p}}\|v\|_{H^1(M)}^2
\]
for every \(R>0\) and \(v\in(H^1(M))^K\), where \(C>0\) is independent of \(R\) and \(v\).
 Since \(u\in L^p(M)\), the coefficient of \(\|v\|_{H^1(M)}^2\) tends to zero as \(R\to\infty\). Choose \(R\) so that this coefficient is at most \(1/2\). For every \(c\geq R+1\), it follows that
\[
B_u(v,v)\geq\frac12\int_M|\nabla v|^2 dM+\left(c-R-\frac12\right)\int_Mv^2 dM\geq\frac12\|v\|_{H^1(M)}^2.
\]
Thus \(B_u\) is coercive. It is also continuous by the Cauchy--Schwarz and H\"older inequalities and the assumed embedding.

For \(f\in(L^2(M))^K\), the functional \(w\mapsto\int_Mfw dM\) is bounded on \((H^1(M))^K\). The Lax--Milgram theorem therefore gives a unique \(v\in(H^1(M))^K\) satisfying \(B_u(v,w)=\int_Mfw dM\) for every \(w\in(H^1(M))^K\). Taking \(w=v\), we obtain
\[
\frac12\|v\|_{H^1(M)}^2\leq B_u(v,v)=\int_Mfv dM\leq\|f\|_{L^2(M)}\|v\|_{H^1(M)},
\]
which gives the asserted estimate.
\end{proof}
Fix \(u\in(H^1(M))^K\) and choose \(c>0\) as in Lemma~\ref{lem:shifted-schrodinger}. Identifying each \(f\in(L^2(M))^K\) with the bounded linear functional \(w\mapsto\int_Mfw dM\) on \((H^1(M))^K\), we denote by \((A_u+c)^{-1}f\) the unique weak solution of \eqref{eq:shifted-schrodinger}. Thus
\begin{equation}\label{eq:Tu-weak}
B_u((A_u+c)^{-1}f,w)=\int_Mfw dM
\end{equation}
for every \(w\in(H^1(M))^K\).

By Lemma~\ref{lem:shifted-schrodinger}, the linear operator \((A_u+c)^{-1}:(L^2(M))^K\longrightarrow(H^1(M))^K\) satisfies \(\|(A_u+c)^{-1}f\|_{H^1(M)}\leq2\|f\|_{L^2(M)}\). Since \((H^1(M))^K\subset(L^2(M))^K\) and \(\|v\|_{L^2(M)}\leq\|v\|_{H^1(M)}\), we may also take \((L^2(M))^K\) as its codomain. With this convention, \((A_u+c)^{-1}:(L^2(M))^K\longrightarrow(L^2(M))^K\) is a bounded linear operator.
\begin{Lemma}\label{lem:resolvent-properties}
The operator \((A_u+c)^{-1}:(L^2(M))^K\to(L^2(M))^K\) is compact, self-adjoint, positive and injective.
\end{Lemma}

\begin{proof}
By Lemma~\ref{lem:shifted-schrodinger}, \((A_u+c)^{-1}\) maps bounded subsets of \((L^2(M))^K\) into bounded subsets of \((H^1(M))^K\). The Rellich--Kondrachov theorem therefore implies its compactness as an operator on \((L^2(M))^K\).

For \(f,h\in(L^2(M))^K\), the symmetry of \(B_u\) and \eqref{eq:Tu-weak} give
\[
\begin{aligned}
\langle(A_u+c)^{-1}f,h\rangle_{L^2(M)}
=B_u((A_u+c)^{-1}f,(A_u+c)^{-1}h)=\langle f,(A_u+c)^{-1}h\rangle_{L^2(M)},
\end{aligned}
\]
so \((A_u+c)^{-1}\) is self-adjoint. Moreover, taking \(w=(A_u+c)^{-1}f\) in \eqref{eq:Tu-weak} and using coercivity, we obtain
\[
\langle(A_u+c)^{-1}f,f\rangle_{L^2(M)}=B_u((A_u+c)^{-1}f,(A_u+c)^{-1}f)\geq\frac12\|(A_u+c)^{-1}f\|_{H^1(M)}^2\geq0.
\]
Thus the operator is positive. Finally, if \((A_u+c)^{-1}f=0\), then \(v=0\) solves \eqref{eq:shifted-schrodinger}, so \(f=0\). Thus \((A_u+c)^{-1}\) is injective.
\end{proof}

From now on, we assume that \((L^2(M))^K\) is infinite-dimensional to avoid treating the finite-dimensional case separately. By Lemma~\ref{lem:resolvent-properties} and the spectral theorem for compact self-adjoint operators, the nonzero elements of \(\sigma_K((A_u+c)^{-1})\) are positive eigenvalues. Ordered decreasingly and counted with their \(K\)-invariant multiplicities, they form a sequence
\(
\alpha_0(u)\geq\alpha_1(u)\geq\alpha_2(u)\geq\cdots>0,
\)
with \(\alpha_j(u)\to0\) as \(j\to\infty\). The corresponding eigenvalues of \(A_u\) are
\begin{equation}\label{eq:Au-eigenvalues}
\mu_j(u)=\alpha_j(u)^{-1}-c.
\end{equation}
These are all the \(K\)-invariant eigenvalues of \(A_u\), ordered increasingly and counted with their \(K\)-invariant multiplicities.

When \(u=0\), the operator \(A_0\) is \(-\Delta\), and we write \(\mu_j=\mu_j(0)\). Its smallest eigenvalue is \(0\), whose eigenspace consists precisely of the constant functions because \(M\) is connected. Thus
\(
0=\mu_0<\mu_1\leq\mu_2\leq\cdots.
\)

\begin{Lemma}\label{lem:shift-independence}
The numbers \(\mu_j(u)\) defined by \eqref{eq:Au-eigenvalues} are independent of the choice of \(c\).
\end{Lemma}
\begin{proof}
Let \(\widetilde c\) be another sufficiently large shift as in Lemma~\ref{lem:shifted-schrodinger}, and let \(\widetilde\alpha_j(u)\) denote the nonzero eigenvalues of \((A_u+\widetilde c)^{-1}\), ordered decreasingly and counted with their \(K\)-invariant multiplicities. Interchanging the two shifts if necessary, we may assume that \(\widetilde c\geq c\).

We first show that
\begin{equation}\label{eq:eigeq}
\bV_{(A_u+c)^{-1}}^K(\alpha_j(u))=\bV_{(A_u+\widetilde c)^{-1}}^K\bigl((\alpha_j(u)^{-1}+\widetilde c-c)^{-1}\bigr).
\end{equation}
Since \(\alpha_j(u)^{-1}+\widetilde c-c\neq0\), the weak formulation of \eqref{eq:shifted-schrodinger} and uniqueness in Lemma~\ref{lem:shifted-schrodinger} give
\[
\begin{aligned}
(A_u+c)^{-1}v=\alpha_j(u)v
&\Longleftrightarrow (A_u+c)v=\alpha_j(u)^{-1}v\\
&\Longleftrightarrow (A_u+\widetilde c)v=(\alpha_j(u)^{-1}+\widetilde c-c)v\\
&\Longleftrightarrow (A_u+\widetilde c)^{-1}v=(\alpha_j(u)^{-1}+\widetilde c-c)^{-1}v
\end{aligned}
\]
for \(v\in(H^1(M))^K\). This proves \eqref{eq:eigeq}.
Therefore
\[
\sigma_K((A_u+\widetilde c)^{-1})\setminus\{0\}=\bigl\{(\alpha_j(u)^{-1}+\widetilde c-c)^{-1}:\alpha_j(u)\in\sigma_K((A_u+c)^{-1})\setminus\{0\}\bigr\}.
\]
By \eqref{eq:eigeq}, the eigenvalues \(\alpha_j(u)\) and \((\alpha_j(u)^{-1}+\widetilde c-c)^{-1}\) have the same \(K\)-invariant multiplicity. Moreover, since \(t\mapsto(t^{-1}+\widetilde c-c)^{-1}\) is strictly increasing on \((0,\infty)\), the numbers \((\alpha_j(u)^{-1}+\widetilde c-c)^{-1}\) are ordered decreasingly as \(j\) increases. Hence \(\widetilde\alpha_j(u)=(\alpha_j(u)^{-1}+\widetilde c-c)^{-1}\) for every \(j\geq0\), and therefore \( \widetilde\alpha_j(u)^{-1}-\widetilde c=\alpha_j(u)^{-1}-c. \)
\end{proof}

We next establish continuity of these ordered eigenvalues with respect to \(u\).
\begin{Lemma}\label{lem:eigenvalue-continuity}
For every \(j\geq0\), the map \((H^1(M))^K\ni u\longmapsto\mu_j(u)\in\bR\) is continuous.
\end{Lemma}

\begin{proof}
Let \(u_n\to u\) in \((H^1(M))^K\). We first choose a common shift \(c>0\) for which the inverse operators \((A_{u_n}+c)^{-1}\) converge to \((A_u+c)^{-1}\) in operator norm on \((L^2(M))^K\).

Choose \(c>0\) as in Lemma~\ref{lem:shifted-schrodinger} for \(u\), so that \(B_u(v,v)\geq\frac12\|v\|_{H^1(M)}^2\). H\"older's inequality and the assumed continuous embedding of \((H^1(M))^K\) into \(L^p(M)\) give a constant \(C>0\), independent of \(n\), such that
\begin{equation}\label{eq:bilinear-form-difference}
|B_{u_n}(v,w)-B_u(v,w)|
\leq C\bigl\||u_n|^{p-2}-|u|^{p-2}\bigr\|_{L^{p/(p-2)}(M)}\|v\|_{H^1(M)}\|w\|_{H^1(M)}
\end{equation}
for every \(v,w\in(H^1(M))^K\).
Since \(u_n\to u\) in \((H^1(M))^K\), the same embedding gives \(u_n\to u\) in \(L^p(M)\), and hence \(|u_n|^{p-2}\to|u|^{p-2}\) in \(L^{p/(p-2)}(M)\).
Applying \eqref{eq:bilinear-form-difference} with \(w=v\), we therefore obtain
\begin{equation}\label{eq:uniform-coercivity}
B_{u_n}(v,v)\geq\frac14\|v\|_{H^1(M)}^2
\end{equation}
for every \(v\in(H^1(M))^K\) and all sufficiently large \(n\). The Lax--Milgram theorem thus defines \((A_{u_n}+c)^{-1}\) with this same shift \(c\).

Fix \(f\in(L^2(M))^K\). By \eqref{eq:Tu-weak},
\[
B_{u_n}((A_{u_n}+c)^{-1}f,w)=\int_Mfw dM=B_u((A_u+c)^{-1}f,w)
\]
for every \(w\in(H^1(M))^K\). From this and using linearity,
\begin{equation}\label{eq:resolvent-difference-identity}
B_{u_n}\bigl((A_{u_n}+c)^{-1}f-(A_u+c)^{-1}f,w\bigr)=B_u((A_u+c)^{-1}f,w)-B_{u_n}((A_u+c)^{-1}f,w)
\end{equation}
for every \(w\in(H^1(M))^K\). Choosing \(w=(A_{u_n}+c)^{-1}f-(A_u+c)^{-1}f\) in \eqref{eq:resolvent-difference-identity} and applying \eqref{eq:uniform-coercivity}, \eqref{eq:bilinear-form-difference}, and Lemma~\ref{lem:shifted-schrodinger}, we obtain
\[
\bigl\|(A_{u_n}+c)^{-1}f-(A_u+c)^{-1}f\bigr\|_{H^1(M)}
\leq 8C\bigl\||u_n|^{p-2}-|u|^{p-2}\bigr\|_{L^{p/(p-2)}(M)}\|f\|_{L^2(M)}.
\]
Since \(\|v\|_{L^2(M)}\leq\|v\|_{H^1(M)}\), the same bound holds with the \(L^2(M)\)-norm on the left. Taking the supremum over \(f\in(L^2(M))^K\) with \(\|f\|_{L^2(M)}\leq1\), and using \(|u_n|^{p-2}\to|u|^{p-2}\) in \(L^{p/(p-2)}(M)\), proves that \((A_{u_n}+c)^{-1}\to(A_u+c)^{-1}\) in operator norm on \((L^2(M))^K\).

Hence, by the max--min principle, the ordered eigenvalues of these compact positive self-adjoint operators satisfy \(\alpha_j(u_n)\to\alpha_j(u)\) for every \(j\geq0\). Since \(\alpha_j(u)>0\), \eqref{eq:Au-eigenvalues} and Lemma~\ref{lem:shift-independence} yield
\(
\mu_j(u_n)=\alpha_j(u_n)^{-1}-c\longrightarrow\alpha_j(u)^{-1}-c=\mu_j(u).
\)
\end{proof}

Under the simplicity assumption of Theorem~\ref{thm:spectral-index}, every point of \(\overline{\cN_K}\) can be written uniquely as \((u,\mu_j(u))\) for some \(j\geq0\). Indeed, if \(u\neq0\), then \(A_uu=\lambda u\), while for \(u=0\) the assertion follows from Lemma~\ref{lem:index-trivial}. In both cases, simplicity ensures the uniqueness of \(j\). We therefore define the index \(\cI:\overline{\cN_K}\to\mathbb N_0\) by \(\cI(u,\mu_j(u))=j\).

\begin{Lemma}\label{lem:index-continuity}
The index
\(
    \cI:\overline{\cN_K}\to\mathbb N_0
\)
is continuous.
\end{Lemma}

\begin{proof}
Let \((u_n,\mu_{j_n}(u_n))\to(u,\mu_j(u))\) in \(\overline{\cN_K}\). We will show that \(\cI(u_n,\mu_{j_n}(u_n))=j_n\to j=\cI(u,\mu_j(u))\). It suffices to prove that \(j_n=j\) for all sufficiently large \(n\).

By the ordering and simplicity of the eigenvalues, \(\mu_j(u)<\mu_{j+1}(u)\). Since \(\mu_{j_n}(u_n)\to\mu_j(u)\) and, by Lemma~\ref{lem:eigenvalue-continuity}, \(\mu_{j+1}(u_n)\to\mu_{j+1}(u)\), we obtain \(\mu_{j_n}(u_n)<\mu_{j+1}(u_n)\) for all sufficiently large \(n\). The ordering of the eigenvalues therefore gives \(j_n\leq j\).

If \(j\geq1\), the same argument applied to \(\mu_{j-1}(u)<\mu_j(u)\) gives \(\mu_{j-1}(u_n)<\mu_{j_n}(u_n)\) for all sufficiently large \(n\), hence \(j_n\geq j\). For \(j=0\), this lower bound holds because \(j_n\in\mathbb N_0\). Thus \(j_n=j\) for all sufficiently large \(n\), proving the continuity of \(\cI\).
\end{proof}

Since \(\mathbb N_0\) is discrete, Lemma~\ref{lem:index-continuity} implies that \(\cI\) is locally constant. Moreover, the continuous image of a connected set is connected, whereas the only connected subsets of \(\mathbb N_0\) are singletons. Hence \(\cI\) is constant on every connected subset of \(\overline{\cN_K}\). We can now prove Theorem~\ref{thm:spectral-index}.

\begin{proof}[Proof of Theorem~\ref{thm:spectral-index}]
By Lemma~\ref{lem:index-trivial}, every point of \(\overline{\cN_K}\cap\cT_K\) has the form \((0,\mu_j)\) for some \(j\geq0\). Since \(\cI\) is constant on connected components and \(\cI(0,\mu_j)=j\), a component containing both \((0,\mu_j)\) and \((0,\mu_k)\) must satisfy \(j=k\). Thus the two points coincide.
\end{proof}

\subsection{Cohomogeneity one actions}\label{subsec:cohom1}

We now verify the simplicity assumption of Theorem~\ref{thm:spectral-index} for cohomogeneity one actions with interval orbit space.

For \(x\in M\), denote its isotropy subgroup by \(K_x=\{g\in K:gx=x\}\) and its orbit by \(K(x)=\{gx:g\in K\}\). By the Principal Orbit Theorem \cite[Theorem~3.82]{AlexandrinoBettiol2015}, there exists a subgroup \(H\subset K\), unique up to conjugation, such that
\(
M_{\mathrm{princ}}=\{x\in M:K_x\text{ is conjugate to }H\}
\)
is open and dense in \(M\). Its orbits are called principal, and their codimension in \(M\) is the cohomogeneity of the action.

Throughout this subsection, we assume that the action has cohomogeneity one. Since \(M\) is compact and connected, \(M/K\) is either a circle or a closed interval; see \cite[Section~6.3]{AlexandrinoBettiol2015}. We restrict to the interval case: the endpoints correspond to the two nonprincipal orbits, and interior points correspond to principal orbits. The space \((L^2(M))^K\) is isomorphic to \(L^2(M/K)\) and is therefore infinite-dimensional.

Let \(O_-\) and \(O_+\) be the two nonprincipal orbits and set \(l=d(O_-,O_+)\). Since these orbits are compact, there exists a minimizing geodesic \(\gamma:[0,l]\to M\) joining them, with \(\gamma(0)\in O_-\), \(\gamma(l)\in O_+\), and \(|\gamma'(t)|=1\). This geodesic is orthogonal to every orbit it meets; see the discussion following \cite[Remark~3.79]{AlexandrinoBettiol2015}. Moreover, it meets every \(K\)-orbit exactly once, with \(\gamma(t)\in M_{\mathrm{princ}}\) for \(t\in(0,l)\); see the discussion preceding Proposition~6.41 of \cite{AlexandrinoBettiol2015}.

Thus a \(K\)-invariant function is determined by its values along \(\gamma\). On \(M_{\mathrm{princ}}\), the metric has the form \(dt^2+g_t\), where \(g_t\) is the metric induced on \(K(\gamma(t))\); see \cite[Proposition~6.41]{AlexandrinoBettiol2015}. For \(t\in(0,l)\), let \(\rho(t)\) denote the Riemannian volume of this orbit. Then \(\rho\) is smooth and strictly positive on \((0,l)\).

For \(0<a<b<l\), set \(\Omega_{a,b}=\bigcup_{t\in(a,b)}K(\gamma(t))\). For a \(K\)-invariant function \(f\) on \(M_{\mathrm{princ}}\), write \(\widetilde f=f\circ\gamma|_{(0,l)}\). If \(f\) is defined only on \(\Omega_{a,b}\), the same notation means \(\widetilde f=f\circ\gamma|_{(a,b)}\).

\begin{Lemma}\label{lem:invariant-integral-identities}
For every \(f\in L^1(\Omega_{a,b})^K\),
\begin{equation}\label{eq:orbit-integration}
\int_{\Omega_{a,b}}f dM=\int_a^b\widetilde f(t)\rho(t)dt.
\end{equation}
If \(v,w\in C^\infty(M)^K\), then
\begin{equation}\label{eq:local-h1-reduction}
\int_{\Omega_{a,b}}\bigl(|\nabla v|^2+v^2\bigr)dM=\int_a^b\bigl(|\widetilde v'(t)|^2+|\widetilde v(t)|^2\bigr)\rho(t)dt
\end{equation}
and
\begin{equation}\label{eq:local-gradient-reduction}
\int_{\Omega_{a,b}}\langle\nabla v,\nabla w\rangle dM=\int_a^b\widetilde v'(t)\widetilde w'(t)\rho(t)dt.
\end{equation}
\end{Lemma}
\begin{proof}
The metric decomposition \(dt^2+g_t\) and the definition of \(\rho\) give \eqref{eq:orbit-integration}.

Let \(v\in C^\infty(M)^K\). Since \(v\) is constant on the orbits, its gradient is normal to them. The normal space to \(K(\gamma(t))\) is spanned by \(\gamma'(t)\), so \(\nabla v(\gamma(t))=a(t)\gamma'(t)\) for some function \(a:(0,l)\to\bR\). Taking the inner product with \(\gamma'(t)\) and using \(|\gamma'(t)|=1\), we obtain
\(
a(t)=\langle\nabla v(\gamma(t)),\gamma'(t)\rangle=\widetilde v'(t),
\)
where the last equality follows from the chain rule. Thus \(\nabla v(\gamma(t))=\widetilde v'(t)\gamma'(t)\), and consequently \(|\nabla v(\gamma(t))|^2=|\widetilde v'(t)|^2\).   Applying the same argument to \(w\in C^\infty(M)^K\), we also obtain \(\langle\nabla v(\gamma(t)),\nabla w(\gamma(t))\rangle=\widetilde v'(t)\widetilde w'(t)\).

Formula \eqref{eq:orbit-integration}, applied to \(|\nabla v|^2+v^2\) and to \(\langle\nabla v,\nabla w\rangle\), now gives \eqref{eq:local-h1-reduction} and \eqref{eq:local-gradient-reduction}, respectively.
\end{proof}

The integral identity \eqref{eq:local-h1-reduction}, together with smooth approximation, gives the following reduction.

\begin{Lemma}\label{lem:invariant-h1-reduction}
If \(v\in(H^1(M))^K\), then \(\widetilde v\in H^1_{\mathrm{loc}}(0,l)\).
\end{Lemma}

\begin{proof}
Let \(v\in(H^1(M))^K\). Choose \(v_n\in C^\infty(M)^K\) such that \(v_n\to v\) in \(H^1(M)\). For every \(0<a<b<l\), applying \eqref{eq:local-h1-reduction} to \(v_n-v_m\) gives
\[
\|\widetilde v_n-\widetilde v_m\|_{H^1(a,b)}^2\leq\frac{1}{\min_{t\in[a,b]}\rho(t)}\|v_n-v_m\|_{H^1(M)}^2.
\]
Thus \((\widetilde v_n)\) converges in \(H^1(a,b)\). By the continuous embedding of \(H^1(a,b)\) into \(C([a,b])\), it also converges uniformly on \([a,b]\). We may therefore define \(\widetilde v(t)=\lim_{n\to\infty}\widetilde v_n(t)\) for every \(t\in(0,l)\). On each \([a,b]\subset(0,l)\), uniform convergence implies convergence in \(L^2(a,b)\), so uniqueness of the limit shows that \(\widetilde v\) agrees almost everywhere with the \(H^1(a,b)\)-limit. Consequently, \(\widetilde v_n\to\widetilde v\) in \(H^1(a,b)\) for every \(0<a<b<l\), and therefore \(\widetilde v\in H^1_{\mathrm{loc}}(0,l)\).

Define \(\widehat v:M_{\mathrm{princ}}\to\bR\) by \(\widehat v(k\gamma(t))=\widetilde v(t)\). Since \(\gamma\) meets each principal orbit exactly once, this function is well-defined and \(K\)-invariant. To prove that it represents \(v\), we first verify equality almost everywhere on each \(\Omega_{a,b}\). By \eqref{eq:orbit-integration},
\[
\|v_n-\widehat v\|_{L^2(\Omega_{a,b})}^2=\int_a^b|\widetilde v_n(t)-\widetilde v(t)|^2\rho(t)dt\leq\left(\max_{t\in[a,b]}\rho(t)\right)\|\widetilde v_n-\widetilde v\|_{L^2(a,b)}^2\longrightarrow0.
\]
Since \(v_n\to v\) in \(L^2(\Omega_{a,b})\), uniqueness of the limit gives \(\widehat v=v\) almost everywhere on \(\Omega_{a,b}\).

Now let \(E=\{x\in M_{\mathrm{princ}}:\widehat v(x)\neq v(x)\}\). The equality \(\widehat v=v\) almost everywhere on \(\Omega_{a,b}\) implies that \(E\cap\Omega_{a,b}\) has measure zero for every \(0<a<b<l\). Since \(M_{\mathrm{princ}}=\Omega_{l/3,2l/3}\cup\Omega_{l/4,3l/4}\cup\cdots\), we have \(E=(E\cap\Omega_{l/3,2l/3})\cup(E\cap\Omega_{l/4,3l/4})\cup\cdots\). Thus \(E\) has measure zero, being a countable union of sets of measure zero. Hence \(\widehat v\) is a \(K\)-invariant representative of \(v\) on \(M_{\mathrm{princ}}\), with \(\widetilde v\in H^1_{\mathrm{loc}}(0,l)\).
\end{proof}

The function \(\widetilde v\) constructed in the proof is continuous, being a locally uniform limit of continuous functions. We now express the weak eigenvalue equation for a \(K\)-invariant eigenfunction of \(A_u\) in terms of \(\widetilde v\) and use the continuity of \(\widetilde u\) and \(\widetilde v\) to show that the resulting equation holds classically.

\begin{Lemma}\label{lem:invariant-eigenfunction-reduction}
Let \(u\in(H^1(M))^K\) and \(v\in\bV_{A_u}^K(\mu_j(u))\). Then the corresponding continuous function \(\widetilde v\) belongs to \(C^2(0,l)\) and satisfies
\begin{equation}\label{eq:reduced-schrodinger-derivative}
\bigl(\rho(t)\widetilde v'(t)\bigr)'=-\rho(t)\bigl(\mu_j(u)+|\widetilde u(t)|^{p-2}\bigr)\widetilde v(t)
\end{equation}
for every \(t\in(0,l)\).
\end{Lemma}

\begin{proof}
Fix \(0<a<b<l\). We first extend \eqref{eq:local-gradient-reduction} to \(v\in(H^1(M))^K\). Let \(v_n\in C^\infty(M)^K\) be such that \(v_n\to v\) in \(H^1(M)\). As shown in the proof of Lemma~\ref{lem:invariant-h1-reduction}, we also have \(\widetilde v_n\to\widetilde v\) in \(H^1(a,b)\). By the Cauchy--Schwarz inequality and the boundedness of \(\rho\) on \([a,b]\), passing to the limit in \eqref{eq:local-gradient-reduction} applied to \(v_n\) and \(w\in C^\infty(M)^K\) gives
\begin{equation}\label{eq:sobolev-gradient-reduction}
\int_{\Omega_{a,b}}\langle\nabla v,\nabla w\rangle dM=\int_a^b\widetilde v'(t)\widetilde w'(t)\rho(t)dt
\end{equation}
for every \(w\in C^\infty(M)^K\).

Now let \(u\in(H^1(M))^K\) and \(v\in\bV_{A_u}^K(\mu_j(u))\). We show that \(\widetilde v\) satisfies \eqref{eq:reduced-schrodinger-derivative} in the sense of distributions on \((a,b)\). Since \(v\) is an eigenfunction of \(A_u\) corresponding to \(\mu_j(u)\), it satisfies
\begin{equation}\label{eq:weak-eigenfunction}
\int_M\langle\nabla v,\nabla w\rangle dM-\int_M|u|^{p-2}vw dM=\mu_j(u)\int_Mvw dM
\end{equation}
for every \(w\in(H^1(M))^K\).

Let \(\widetilde w\in C_c^\infty(a,b)\), and define \(w\in C^\infty(M)^K\) by \(w(k\gamma(t))=\widetilde w(t)\) on \(\Omega_{a,b}\) and \(w=0\) outside \(\Omega_{a,b}\). Both \(w\) and \(\nabla w\) vanish outside \(\Omega_{a,b}\), so all integrals in \eqref{eq:weak-eigenfunction} can be taken over \(\Omega_{a,b}\). Applying \eqref{eq:orbit-integration} and \eqref{eq:sobolev-gradient-reduction} gives
\begin{equation}\label{eq:reduced-schrodinger-weak}
\int_a^b\rho(t)\widetilde v'(t)\widetilde w'(t)dt=\int_a^b\rho(t)\bigl(\mu_j(u)+|\widetilde u(t)|^{p-2}\bigr)\widetilde v(t)\widetilde w(t)dt.
\end{equation}

By Lemma~\ref{lem:invariant-h1-reduction}, \(\widetilde v\in H^1_{\mathrm{loc}}(0,l)\), so \(\rho\widetilde v'\in L^2(a,b)\). Since \eqref{eq:reduced-schrodinger-weak} holds for every \(\widetilde w\in C_c^\infty(a,b)\), the distributional derivative of \(\rho\widetilde v'\) equals \(-\rho(\mu_j(u)+|\widetilde u|^{p-2})\widetilde v\). This function is continuous, so \(\rho\widetilde v'\) has a \(C^1\) representative on \((a,b)\). Since \(\rho\) is smooth and strictly positive, it follows that \(\widetilde v\in C^2(a,b)\), and the equation holds classically there. As \(0<a<b<l\) was arbitrary, \(\widetilde v\in C^2(0,l)\), and \eqref{eq:reduced-schrodinger-derivative} holds for every \(t\in(0,l)\).
\end{proof}

We use this reduction to prove the simplicity of the \(K\)-invariant spectrum of \(A_u\).

\begin{Theorem}\label{thm:schrodinger-simplicity}
Let \(M\) be a compact connected Riemannian manifold and let \(K\) be a compact Lie group acting isometrically on \(M\). Assume that \(M/K\) is homeomorphic to a closed interval and that \((H^1(M))^K\) is continuously embedded in \(L^p(M)\). Then, for every \(u\in(H^1(M))^K\), every \(K\)-invariant eigenvalue of \(A_u=-\Delta-|u|^{p-2}\) is simple.
\end{Theorem}

\begin{proof}
Fix \(u\in(H^1(M))^K\) and let \(v_1,v_2\in\bV_{A_u}^K(\mu_j(u))\setminus\{0\}\). We show that \(v_1\) and \(v_2\) are linearly dependent. On \(M_{\mathrm{princ}}\), we use their continuous \(K\)-invariant representatives given by Lemma~\ref{lem:invariant-h1-reduction}. By Lemma~\ref{lem:invariant-eigenfunction-reduction}, the corresponding functions \(\widetilde v_1,\widetilde v_2\) belong to \(C^2(0,l)\) and satisfy \eqref{eq:reduced-schrodinger-derivative}.
Define their Wronskian by
\(
W(t)=\widetilde v_1(t)\widetilde v_2'(t)-\widetilde v_1'(t)\widetilde v_2(t).
\)
Equation \eqref{eq:reduced-schrodinger-derivative} gives
\(
(\rho W)'=\widetilde v_1(\rho\widetilde v_2')'-\widetilde v_2(\rho\widetilde v_1')'=0.
\)
Consequently, there exists \(C\in\bR\) such that
\(\rho(t)W(t)=C\)
for every \(t\in(0,l)\).

To show that \(C=0\), let \(f\in C^\infty(M)^K\). Then \(fv_1,fv_2\in(H^1(M))^K\). Using \(fv_2\) in the weak equation \(A_uv_1=\mu_j(u)v_1\) and \(fv_1\) in the weak equation \(A_uv_2=\mu_j(u)v_2\), and subtracting, we obtain
\begin{equation}\label{eq:wronskian-weak-identity}
0=\int_M\langle\nabla v_1,\nabla(fv_2)\rangle dM-\int_M\langle\nabla v_2,\nabla(fv_1)\rangle dM.
\end{equation}

Now choose \(f\) such that \(f=1\) near \(O_-\) and \(f=0\) near \(O_+\). By the product rule, the difference of the integrands in \eqref{eq:wronskian-weak-identity} is \(\langle v_2\nabla v_1-v_1\nabla v_2,\nabla f\rangle\), which vanishes near both nonprincipal orbits. By Lemma~\ref{lem:invariant-eigenfunction-reduction}, \(\widetilde v_1,\widetilde v_2\in C^2(0,l)\). Reasoning as in the proof of Lemma~\ref{lem:invariant-integral-identities}, we obtain \(\nabla v_i(\gamma(t))=\widetilde v_i'(t)\gamma'(t)\) for \(i=1,2\).
Thus \eqref{eq:orbit-integration} gives
\[
0=\int_0^l\bigl(\widetilde v_2(t)\widetilde v_1'(t)-\widetilde v_1(t)\widetilde v_2'(t)\bigr)\widetilde f'(t)\rho(t)dt=-\int_0^l\widetilde f'(t)\rho(t)W(t)dt.
\]
Since \(\rho(t)W(t)=C\), while \(\widetilde f=1\) near \(0\) and \(\widetilde f=0\) near \(l\), we obtain
\(
0=-C\int_0^l\widetilde f'(t)dt=C.
\)
Thus \(C=0\), and positivity of \(\rho\) implies that \(W=0\) on \((0,l)\).
Therefore \(\widetilde v_1=c\widetilde v_2\) on \((0,l)\) for some \(c\in\bR\). Since \(v_1\) and \(v_2\) are constant on each orbit, it follows that \(v_1=cv_2\) on \(M_{\mathrm{princ}}\). Its complement consists of the two nonprincipal orbits and has Riemannian measure zero, so \(v_1=cv_2\) almost everywhere on \(M\). Hence~\(\dim\bV_{A_u}^K(\mu_j(u))=1\).
\end{proof}

Theorem~\ref{thm:spectral-index} therefore applies to cohomogeneity one actions whose orbit space is a closed interval. For these actions, the index \(\cI\) also distinguishes solutions of one sign from sign-changing solutions.

\begin{Proposition}\label{prop:sign-changing}
Under the assumptions of Theorem~\ref{thm:schrodinger-simplicity}, let \((u,\mu_j(u))\in\cN_K\). If \(j=0\), then \(u\) is either positive or negative almost everywhere on \(M\). If \(j\geq1\), then \(u\) changes sign on \(M\).

Consequently, if \(\mathcal C\) is a connected component of \(\overline{\cN_K}\) containing \((0,\mu_j)\), then every nontrivial solution in \(\mathcal C\) has one sign when \(j=0\), and changes sign when \(j\geq1\).
\end{Proposition}

\begin{proof}
Fix \((u,\mu_j(u))\in\cN_K\). Then \(u\) is a nonzero eigenfunction of \(A_u\) corresponding to \(\mu_j(u)\). To determine its sign, we first construct an eigenfunction corresponding to \(\mu_0(u)\) that is positive almost everywhere on \(M\).

By the standard variational characterization,
\begin{equation}\label{eq:first-eigenvalue-characterization}
\mu_0(u)=\min_{v\in(H^1(M))^K\setminus\{0\}}\frac{\int_M\left(|\nabla v|^2-|u|^{p-2}v^2\right)dM}{\int_Mv^2dM},
\end{equation}
and the minimizers are precisely the nonzero eigenfunctions corresponding to \(\mu_0(u)\). Choose such an eigenfunction \(v_0\). By \cite[Proposition~3.49]{Aubin1982}, we have \(|v_0|\in H^1(M)\) and \(\left|\nabla|v_0|\right|=|\nabla v_0|\) almost everywhere. Moreover, \(|v_0|\) is \(K\)-invariant. Thus \(|v_0|\) attains the same minimum in \eqref{eq:first-eigenvalue-characterization} and also belongs to \(\bV_{A_u}^K(\mu_0(u))\). Replacing \(v_0\) by \(|v_0|\), we may assume that \(v_0\geq0\) almost everywhere.

We next show that \(v_0>0\) almost everywhere. By Lemma~\ref{lem:invariant-eigenfunction-reduction}, \(\widetilde v_0\in C^2(0,l)\) satisfies \eqref{eq:reduced-schrodinger-derivative} with eigenvalue \(\mu_0(u)\). Its continuity and the inequality \(v_0\geq0\) almost everywhere imply that \(\widetilde v_0\geq0\) on \((0,l)\). If \(\widetilde v_0(t_0)=0\) for some \(t_0\in(0,l)\), then \(t_0\) is a minimum, so \(\widetilde v_0'(t_0)=0\). Uniqueness for the homogeneous linear equation \eqref{eq:reduced-schrodinger-derivative} with initial conditions \(\widetilde v_0(t_0)=\widetilde v_0'(t_0)=0\) implies that \(\widetilde v_0\equiv0\) on \((0,l)\). The corresponding representative of \(v_0\) would therefore vanish on \(M_{\mathrm{princ}}\). As its complement has measure zero, this contradicts \(v_0\neq0\) in \(H^1(M)\). Hence \(\widetilde v_0>0\) on \((0,l)\), which gives \(v_0>0\) almost everywhere on \(M\).

If \(j=0\), then \(u\) and \(v_0\) correspond to the same eigenvalue \(\mu_0(u)\). By Theorem~\ref{thm:schrodinger-simplicity}, \(u\) is a nonzero multiple of \(v_0\), so \(u\) is either positive or negative almost everywhere.

If \(j\geq1\), then \(\mu_j(u)>\mu_0(u)\). Using \(v_0\) in the weak eigenvalue equation for \(u\), and \(u\) in that for \(v_0\), gives
\[
\mu_j(u)\int_Muv_0 dM
=\int_M\bigl(\langle\nabla u,\nabla v_0\rangle-|u|^{p-2}uv_0\bigr)dM
=\mu_0(u)\int_Muv_0 dM.
\]
Thus \(\int_Muv_0 dM=0\). Since \(u\neq0\) and \(v_0>0\) almost everywhere, this equality excludes both \(u\geq0\) and \(u\leq0\) almost everywhere. Hence \(u\) changes sign.

Finally, let \(\mathcal C\) be a connected component of \(\overline{\cN_K}\) containing \((0,\mu_j)\). By Lemma~\ref{lem:index-continuity}, \(\cI\) is constant on \(\mathcal C\), with value \(\cI(0,\mu_j)=j\). Consequently, every nontrivial solution in \(\mathcal C\) has the form \((w,\mu_j(w))\). Applying the result just proved to \(w\), we conclude that it is either positive or negative almost everywhere if \(j=0\), and changes sign if \(j\geq1\).
\end{proof}

\section{Global bifurcation on selected manifolds}\label{sec:selman}

Throughout this section, \(M\) is a closed connected \(n\)-dimensional Riemannian manifold, \(p>2\), and \(G\) is a compact Lie group acting isometrically on \(M\). For a closed subgroup \(K\subset G\), we give sufficient conditions for unbounded continua of \(K\)-invariant weak solutions of \eqref{eq:main} to bifurcate from the trivial branch. We then apply these conditions to homogeneous spaces and distinguish continua of solutions invariant under different subgroups of \(G\).

\subsection{Unbounded continua}\label{subsec:unbounded-continua}

By Lemma~\ref{lem:index-trivial}, every bifurcation point on \(\cT_K\) has the form \((0,\mu_j)\) for some \(\mu_j\in\sigma_K(-\Delta)\). For each such point, denote by \(\cC_K(\mu_j)\) the connected component of \(\overline{\cN_K}\) containing it. The bifurcation is called global if \(\cC_K(\mu_j)\) is unbounded in \((H^1(M))^K\times\bR\) or meets \(\cT_K\) at another point.

We first recall the improved compactness of Sobolev embeddings in the presence of symmetries. Set
\begin{equation}\label{eq:dK}
d_K=\min\{\dim K(x):x\in M\}.
\end{equation}
Since \(K\) is compact, every \(K\)-orbit is infinite if and only if \(d_K\geq1\). In this case, the following result is due to Hebey and Vaugon; see \cite[Corollary~1]{HebVau}. For \(d_K=0\), it follows from the usual Rellich--Kondrachov theorem restricted to \((H^1(M))^K\).

\begin{Theorem}\label{thm:HV}
If \(n-d_K\leq2\), then \((H^1(M))^K\) is compactly embedded in \(L^q(M)\) for every \(1<q<\infty\). If \(n-d_K>2\), then this embedding is compact for every \(1<q<2(n-d_K)/(n-d_K-2)\).
\end{Theorem}

We combine the compactness in Theorem~\ref{thm:HV} with the Rabinowitz global bifurcation theorem from \cite{Rabinowitz1971}. The separation result in Theorem~\ref{thm:spectral-index} excludes a second intersection with the trivial branch, yielding the following unboundedness result.

\begin{Theorem}\label{thm:unbounded-continuum}
Assume that \(M/K\) is homeomorphic to a closed interval and, if \(n-d_K>2\), that \(p<2(n-d_K)/(n-d_K-2)\). Then, for every \(\mu_j\in\sigma_K(-\Delta)\), \(j\geq0\), the point \((0,\mu_j)\in\cT_K\) is a global bifurcation point. The component \(\cC_K(\mu_j)\) is unbounded in \((H^1(M))^K\times\bR\) and meets \(\cT_K\) only at \((0,\mu_j)\). These components are pairwise disjoint. Moreover, every nontrivial solution in \(\cC_K(\mu_j)\) is either positive or negative almost everywhere on \(M\) if \(j=0\), and changes sign if \(j\geq1\).
\end{Theorem}

\begin{proof}
By Theorem~\ref{thm:HV}, the embedding of \((H^1(M))^K\) into \(L^p(M)\) is compact and therefore continuous. Theorem~\ref{thm:schrodinger-simplicity} implies that \(A_u\) has simple \(K\)-invariant spectrum for every \(u\in(H^1(M))^K\). In particular, every \(\mu_j\in\sigma_K(-\Delta)\) is simple. The Rabinowitz global bifurcation theorem from \cite{Rabinowitz1971} therefore shows that \((0,\mu_j)\in\cT_K\) is a global bifurcation point: its component \(\cC_K(\mu_j)\) is either unbounded or meets \(\cT_K\) at another point.

By Theorem~\ref{thm:spectral-index}, the only point of \(\cT_K\) in \(\cC_K(\mu_j)\) is \((0,\mu_j)\), so the component is unbounded. Components corresponding to different eigenvalues are therefore distinct and hence disjoint. Finally, Proposition~\ref{prop:sign-changing} gives the stated conclusions about the signs of their nontrivial solutions.
\end{proof}

\begin{Remark}
The two possibilities in the Rabinowitz alternative are not mutually exclusive: an unbounded component may also meet the trivial branch at another eigenvalue. Theorem~\ref{thm:unbounded-continuum} excludes this possibility, since \(\cC_K(\mu_j)\cap\cT_K=\{(0,\mu_j)\}\).
\end{Remark}
We next verify the geometric assumptions of Theorem~\ref{thm:unbounded-continuum} for actions on compact homogeneous spaces, beginning with the condition \(d_K>0\).

Let \(G\) be a compact Lie group, and let \(H,K\in\sub(G)\) be connected. Consider the action of \(K\) on \(G/H\) given by \(k\cdot(gH)=kgH\). For this action, positivity of \(d_K\) can be expressed directly in terms of the subgroups \(H\) and \(K\).

\begin{Lemma}\label{lem:positive-dimensional-orbits}
Let \(G\) be a compact Lie group, and let \(K,H\subset G\) be connected closed subgroups. For the action of \(K\) on \(G/H\), one has \(d_K>0\) if and only if \((G/H)^K=\varnothing\), or equivalently, if no conjugate of \(K\) is contained in \(H\).
\end{Lemma}

\begin{proof}
Since \(K\) is connected, every orbit is connected and therefore has dimension zero if and only if it consists of a single point. Thus \(d_K=0\) if and only if the action has a fixed point.
A point \(gH\in G/H\) is fixed by \(K\) if and only if \(kgH=gH\) for every \(k\in K\), or equivalently \(g^{-1}Kg\subset H\). Hence a fixed point exists precisely when some conjugate of \(K\) is contained in \(H\).
\end{proof}

We now assume additionally that \(G\) is connected and semisimple, and show that then a cohomogeneity one action of \(K\) on \(G/H\) has interval orbit space. For this purpose, consider the action of \(K\times H\) on \(G\) given by \((k,h)\cdot g=kgh^{-1}\). The map
\(
G/(K\times H)\longrightarrow(G/H)/K,
\)
defined by \((K\times H)(g)\mapsto K(gH)\), is a homeomorphism. In particular, the two actions have the same cohomogeneity.

\begin{Lemma}\label{lem:orbit-spaces-interval}
Let \(G\) be a compact connected semisimple Lie group, and let \(K,H\subset G\) be connected closed subgroups. Assume that the action of \(K\times H\) on \(G\), given by \((k,h)\cdot g=kgh^{-1}\), has cohomogeneity one. Then \((G/H)/K\) is homeomorphic to \([0,1]\).
\end{Lemma}

\begin{proof}
By the structure theorem for compact cohomogeneity one actions, see \cite[Section~6.3]{AlexandrinoBettiol2015}, the orbit space \(G/(K\times H)\) is homeomorphic either to a circle or to a closed interval. Suppose that it is homeomorphic to \(S^1\), and let \(\pi:G\to S^1\) be the corresponding orbit projection.

Since \(G\) is path connected and every orbit of \(K\times H\) is connected, it follows from \cite[Corollary~II.6.3]{Bredon} that the induced homomorphism \(\pi_*\colon\pi_1(G,g)\to\pi_1(S^1,\pi(g))\) is surjective. However, \(\pi_1(G,g)\) is finite by Weyl's theorem, see \cite[Theorem~4.69]{KnappLie}, whereas \(\pi_1(S^1,\pi(g))\) is isomorphic to \(\mathbb Z\). This is a contradiction. Thus \(G/(K\times H)\), and hence \((G/H)/K\), is homeomorphic to a closed interval.
\end{proof}

Under the assumptions of Lemma~\ref{lem:orbit-spaces-interval}, Theorem~\ref{thm:unbounded-continuum} therefore applies to \(M=G/H\) whenever \(p\) satisfies its exponent assumption. By Lemma~\ref{lem:positive-dimensional-orbits}, the admissible range improves on the usual subcritical range when no conjugate of \(K\) is contained in \(H\).

For symmetric subgroups of a simple group, the condition \(d_K>0\) has a simpler characterization. Recall that an automorphism \(\sigma:G\to G\) is called involutive if \(\sigma^2=\Id\). Denote its fixed-point subgroup by \(G^\sigma=\{g\in G:\sigma(g)=g\}\) and its identity component by \((G^\sigma)_0\). A connected closed subgroup \(H\subset G\) is called symmetric if \(H=(G^\sigma)_0\) for some involutive automorphism \(\sigma\).

We call a connected Lie group simple if its Lie algebra is nonabelian and has no nonzero proper ideals. For such groups, we use the following standard maximality property; see the remark following Definition~1.12 in \cite{Kollross}. For completeness we include its proof in the appendix.

\begin{Lemma}\label{lem:maximal-symmetric-subgroup}
Let \(G\) be a compact connected simple Lie group, and let \(H\subset G\) be a proper connected symmetric subgroup. If \(L\subset G\) is a connected closed subgroup containing \(H\), then either \(L=H\) or \(L=G\).
\end{Lemma}

Together with Lemma~\ref{lem:positive-dimensional-orbits}, this shows that, for proper connected symmetric subgroups, the condition \(d_K>0\) is equivalent to nonconjugacy.

\begin{Lemma}\label{lem:symmetric-positive-dimensional-orbits}
Let \(G\) be a compact connected simple Lie group, and let \(H,K\subset G\) be proper connected symmetric subgroups. For the \(K\)-action on \(G/H\), one has \(d_K>0\) if and only if \(H\) and \(K\) are not conjugate in \(G\).
\end{Lemma}

\begin{proof}
By Lemma~\ref{lem:positive-dimensional-orbits}, \(d_K=0\) if and only if \(gKg^{-1}\subset H\) for some \(g\in G\). Since \(gKg^{-1}\) is a proper connected symmetric subgroup and \(H\) is connected and proper, Lemma~\ref{lem:maximal-symmetric-subgroup} shows that this inclusion holds only when \(gKg^{-1}=H\). Thus \(d_K=0\) if and only if \(H\) and \(K\) are conjugate.
\end{proof}

\begin{Remark}
Let \(H,K\subset G\) satisfy the assumptions of Lemma~\ref{lem:symmetric-positive-dimensional-orbits}, and suppose that they are not conjugate. If the \(K\)-action on \(G/H\) has cohomogeneity one, then \(d_K>0\) and, by Lemma~\ref{lem:orbit-spaces-interval}, its orbit space is a closed interval. Thus, when \(n=\dim(G/H)\geq3\), Theorem~\ref{thm:unbounded-continuum} applies at the critical exponent \(p=2n/(n-2)\) and for some supercritical exponents. The admissible supercritical range is determined by \(d_K\), whose exact values for several classical actions are computed in Lemma~\ref{lem:minimal-orbit-dimensions} in the appendix. The cohomogeneity one actions with \(G\) simple and \(H,K\) proper connected symmetric subgroups are classified in \cite[Theorem~B]{Kollross}. This classification therefore provides a collection of actions to which our results apply whenever \(H\) and \(K\) are not conjugate.
\end{Remark}

\subsection{Geometrically distinct bifurcations}

We now give a condition ensuring that continua obtained in different fixed-point spaces are geometrically distinct, even when they bifurcate from the same eigenvalue.

Two solutions \((u,\lambda)\) and \((v,\lambda)\) of \eqref{eq:main} are called \(G\)-geometrically equivalent if \(v=gu\) for some \(g\in G\). Given closed subgroups \(K,L\subset G\), suppose that \((0,\mu)\) is a bifurcation point in both \((H^1(M))^K\times\bR\) and \((H^1(M))^L\times\bR\). We call the corresponding bifurcations \(G\)-geometrically distinct if there exists a neighborhood of \((0,\mu)\) in \(H^1(M)\times\bR\) in which no nontrivial \(K\)-invariant solution is \(G\)-geometrically equivalent to a nontrivial \(L\)-invariant solution.

Assume throughout this subsection that \(G\) acts transitively on \(M\), so that \((H^1(M))^G\) consists of the constant functions. The following argument adapts the fixed-point-space method used in Section~2 and the proof of Theorem~4.8 of \cite{Marzantowicz2019} to compare solutions invariant under different subgroups.

\begin{Proposition}\label{prop:geometric-separation}
Assume that \(G\) acts transitively on \(M\), and let \(K,L\subset G\) be closed subgroups satisfying
\begin{equation}\label{eq:fixed-point-separation}
\forall g\in G\colon (H^1(M))^K\cap(H^1(M))^{gLg^{-1}}=(H^1(M))^G.
\end{equation}
If \(\mu>0\) and \((0,\mu)\) is a bifurcation point in both \((H^1(M))^K\times\bR\) and \((H^1(M))^L\times\bR\), then the corresponding bifurcations are \(G\)-geometrically distinct.

If, in addition, every nontrivial solution in at least one of the components \(\cC_K(\mu)\) and \(\cC_L(\mu)\) changes sign, then their nontrivial parts are \(G\)-geometrically distinct and
\begin{equation}\label{eq:disjointcont}
\bigl(\cC_K(\mu)\setminus\cT_K\bigr)\cap\bigl(\cC_L(\mu)\setminus\cT_L\bigr)=\varnothing.
\end{equation}
\end{Proposition}

\begin{proof}
Let \((u,\lambda)\) and \((v,\lambda)\) be nontrivial \(L\)- and \(K\)-invariant solutions, respectively. Suppose that they are \(G\)-geometrically equivalent, so \(v=gu\) for some \(g\in G\). Since \(u\) is \(L\)-invariant, \(v\) is \(gLg^{-1}\)-invariant. By \eqref{eq:fixed-point-separation},
\(
v\in(H^1(M))^K\cap(H^1(M))^{gLg^{-1}}=(H^1(M))^G.
\)
Transitivity of \(G\) implies that \(v\) is constant, and therefore \(u=g^{-1}v\) is constant as well. Writing \(u\equiv c\neq0\) in \eqref{eq:main}, we obtain \(\lambda=-|c|^{p-2}<0\).

Thus no nontrivial \(K\)-invariant solution can be \(G\)-geometrically equivalent to a nontrivial \(L\)-invariant solution at a positive parameter. Since \(\mu>0\), we may choose a neighborhood of \((0,\mu)\) in which every point has positive parameter. This proves that the bifurcations are \(G\)-geometrically distinct.

For the second assertion, any two \(G\)-geometrically equivalent nontrivial solutions in the respective components would be constant, contradicting the sign-changing assumption. Thus their nontrivial parts are \(G\)-geometrically distinct and, in particular, disjoint.
\end{proof}

\begin{Remark}
At \(\mu=0\), the constant solutions \(u\equiv c\neq0\), with \(\lambda=-|c|^{p-2}\), converge to \((0,0)\) as \(c\to0\). Since these solutions are \(G\)-invariant, they belong to both \(\cN_K\) and \(\cN_L\). Thus every neighborhood of \((0,0)\) contains common nontrivial solutions, so the corresponding bifurcations are not \(G\)-geometrically distinct.
\end{Remark}

\subsection{Multiplicity on spheres and projective spaces}\label{subsec:multip}

Let \(m\geq4\) and \(2\leq a\leq\lfloor m/2\rfloor\). Endow \(S^{m-1}\), \(\bC P^{m-1}\), and \(\bH P^{m-1}\). The relevant invariant spectra are as follows.
\begin{enumerate}
\item[(E1)]
On \(S^{m-1}\), the eigenvalues of \(-\Delta\) are
\(k(k+m-2)\), \(k\geq0\), by
\cite[Theorem~5.1]{Gurarie}. By
\cite[Lemmas~6.4 and~6.5]{SmoWass1990}, the
\(SO(a)\times SO(m-a)\)-fixed subspace of the \(k\)-th
spherical eigenspace is nonzero precisely when \(k\) is even.
Hence the positive invariant eigenvalues are
\(\mu_j=2j(2j+m-2)\), \(j\geq1\).

\item[(E2)]
On \(\bC P^{m-1}\), the computations in
\cite[Section~7]{BetJulPet2021} show that every eigenvalue
belongs to the \(S(U(a)\times U(m-a))\)-invariant spectrum.
Thus the positive invariant eigenvalues are
\(\mu_j=4j(j+m-1)\), \(j\geq1\). The same computations show
that the positive \(SO(m)\)-invariant eigenvalues are precisely
the \(\mu_j\) with \(j\) even.

\item[(E3)]
On \(\bH P^{m-1}\), the computations in
\cite[Section~7]{BetJulPet2021} show that every eigenvalue
belongs to the \(Sp(a)\times Sp(m-a)\)-invariant spectrum.
Thus the positive invariant eigenvalues are
\(\mu_j=4j(j+2m-1)\), \(j\geq1\). The positive
\(U(m)\)-invariant eigenvalues are precisely the \(\mu_j\)
with \(j\) even.
\end{enumerate}
In each case, \(\mu_j\) denotes the common ordered sequence of positive eigenvalues for the block actions. For the additional projective actions, the corresponding sequence is \(\mu_{2r}\), \(r\geq1\).

\begin{Theorem}\label{thm:multiplicity-block-actions}
Let \(m\geq4\), and let \(M\) be \(S^{m-1}\), \(\bC P^{m-1}\), or \(\bH P^{m-1}\), endowed with its standard metric. Let \(G=SO(m)\), \(SU(m)\), or \(Sp(m)\), respectively. Consider the Yamabe-type equation at the critical exponent
\begin{equation}\label{eq:critical-yamabe}
-\Delta u-\lambda u=|u|^{\frac{4}{n-2}}u,
\end{equation}
where \(n=\dim M\). For every \(j\geq1\), equation \eqref{eq:critical-yamabe} admits at least \(\lfloor m/2\rfloor-1\) unbounded continua bifurcating from \((0,\mu_j)\). If \(M=\bC P^{m-1}\) or \(M=\bH P^{m-1}\) and \(j\) is even, there are at least \(\lfloor m/2\rfloor\) such continua. 
Every nontrivial solution belonging to these continua changes sign, and their nontrivial parts are pairwise \(G\)-geometrically distinct and disjoint.
\end{Theorem}

We first prove an auxiliary lemma.

\begin{Lemma}\label{lem:intersection-block-actions}
Let \(m\geq4\), and call an integer \(a\) admissible if \(2\leq a\leq\lfloor m/2\rfloor\). Consider the following actions:
\begin{enumerate}
\item \(M=S^{m-1}\), \(G=SO(m)\), and \(K_a=SO(a)\times SO(m-a)\);

\item \(M=\bC P^{m-1}\), \(G=SU(m)\), and \(K_a=S(U(a)\times U(m-a))\);

\item \(M=\bH P^{m-1}\), \(G=Sp(m)\), and \(K_a=Sp(a)\times Sp(m-a)\).
\end{enumerate}
In each case, the subgroups \(K_a\), with \(a\) admissible, satisfy \eqref{eq:fixed-point-separation} pairwise.

In the complex case, the same remains true after adjoining \(SO(m)\) to this family. In the quaternionic case, it remains true after adjoining \(U(m)\), where \(U(m)\subset Sp(m)\) is the subgroup of quaternionic unitary matrices with complex entries.
\end{Lemma}
\begin{proof}
For the sphere action with \(m=4\), there is only one admissible block subgroup, so there is nothing to prove. In all remaining cases, \(G\) is simple and all the subgroups under consideration are proper connected symmetric subgroups; see for example \cite[Theorem~B(i)]{Kollross}.

To establish \eqref{eq:fixed-point-separation}, it suffices to prove that \(\overline{\langle K,gLg^{-1}\rangle}=G\) for every pair of distinct subgroups \(K,L\) in each family and every \(g\in G\), since \((H^1(M))^K\cap(H^1(M))^{gLg^{-1}}=(H^1(M))^{\overline{\langle K,gLg^{-1}\rangle}}\). Since \(G\) is simple, Lemma~\ref{lem:maximal-symmetric-subgroup} applies to the subgroups under consideration. We first use it to show that \(\overline{\langle K,gLg^{-1}\rangle}=G\) whenever \(K\) and \(L\) are nonconjugate, and then verify pairwise nonconjugacy within each family.

Suppose that \(K\) and \(L\) are nonconjugate, and fix \(g\in G\). Since \(K\) and \(L\) are connected, \(\overline{\langle K,gLg^{-1}\rangle}\) is a connected closed subgroup containing \(K\). Since \(G\) is simple and \(K\) is a proper connected symmetric subgroup, Lemma~\ref{lem:maximal-symmetric-subgroup} shows that \(\overline{\langle K,gLg^{-1}\rangle}\) equals either \(K\) or \(G\). In the first case, \(gLg^{-1}\subset K\). Applying the same lemma to the symmetric subgroup \(gLg^{-1}\) and the proper connected closed subgroup \(K\), we obtain \(K=gLg^{-1}\), contrary to nonconjugacy. Therefore \(\overline{\langle K,gLg^{-1}\rangle}=G\).

It remains to verify pairwise nonconjugacy within each family. Since \(\dim K_a\) strictly decreases as \(a\) increases in the admissible range, distinct block subgroups are not conjugate. Moreover,
\[
\dim K_a-\dim SO(m)=\frac{(m-2a)^2+m-2}{2}>0
\]
in the complex case, and
\[
\dim K_a-\dim U(m)=(m-2a)^2+m>0
\]
in the quaternionic case. Thus the additional subgroup is not conjugate to any of the corresponding block subgroups. This proves pairwise nonconjugacy and hence \eqref{eq:fixed-point-separation}.
\end{proof}

We can now prove Theorem~\ref{thm:multiplicity-block-actions}.
\begin{proof}[Proof of Theorem~\ref{thm:multiplicity-block-actions}]
We apply Theorem~\ref{thm:unbounded-continuum} to each subgroup in Lemma~\ref{lem:intersection-block-actions}, and then use Proposition~\ref{prop:geometric-separation} to distinguish the resulting components.

We first verify the geometric assumptions of Theorem~\ref{thm:unbounded-continuum}. Except for the sphere case \(m=4\), the actions in Lemma~\ref{lem:intersection-block-actions} have cohomogeneity one by \cite[Theorem~B(i)]{Kollross}. For \(m=4\), the principal orbits of \(SO(2)\times SO(2)\) on \(S^3\) are products of two circles, so this action also has cohomogeneity one. In every case, \(G\) is compact, connected, and semisimple, and both the acting subgroup and the isotropy subgroup defining \(M=G/H\) are connected. Hence Lemma~\ref{lem:orbit-spaces-interval} shows that all these actions have interval orbit space.

To apply Theorem~\ref{thm:unbounded-continuum} at the critical exponent, we next show that \(d_K>0\) for every subgroup \(K\) listed in Lemma~\ref{lem:intersection-block-actions}. For the sphere actions, suppose that \((x,y)\in S^{m-1}\subset\bR^a\oplus\bR^{m-a}\) is fixed by \(K_a=SO(a)\times SO(m-a)\). Then \(Ax=x\) for every \(A\in SO(a)\) and \(By=y\) for every \(B\in SO(m-a)\). Since \(a,m-a\geq2\), the only vector fixed by every rotation in either factor is the zero vector. Hence \(x=y=0\), contradicting \(|x|^2+|y|^2=1\). Thus the action has no fixed points and, since \(K_a\) is connected, Lemma~\ref{lem:positive-dimensional-orbits} gives \(d_{K_a}>0\).

For the projective actions, write \(M=G/H\), where \(G=SU(m)\) and \(H=S(U(1)\times U(m-1))\) in the complex case, while \(G=Sp(m)\) and \(H=Sp(1)\times Sp(m-1)\) in the quaternionic case. In both cases, \(H\) is a proper connected symmetric subgroup of \(G\); see \cite[Theorem~B(i)]{Kollross}. As established in the proof of Lemma~\ref{lem:intersection-block-actions}, every acting subgroup \(K\) is also proper, connected, and symmetric. Since \(G\) is simple, Lemma~\ref{lem:symmetric-positive-dimensional-orbits} reduces the condition \(d_K>0\) to nonconjugacy of \(K\) and \(H\). This follows from \(\dim K<\dim H\).

Thus, in every case, \(d_K>0\), so the critical exponent \(p=2n/(n-2)\), where \(n=\dim M\), satisfies the exponent condition in Theorem~\ref{thm:unbounded-continuum}. By (E1)--(E3), the positive \(K_a\)-invariant eigenvalues are precisely \(\mu_j\), \(j\geq1\), for every admissible \(a\). Applying the theorem at each of these eigenvalues gives an unbounded component \(\cC_{K_a}(\mu_j)\) bifurcating from \((0,\mu_j)\), all of whose nontrivial solutions change sign. In the projective cases, (E2) and (E3) also allow us to apply the theorem to the additional subgroups at every \(\mu_{2r}\), \(r\geq1\). We obtain components \(\cC_{SO(m)}(\mu_{2r})\) in the complex case and \(\cC_{U(m)}(\mu_{2r})\) in the quaternionic case, with the same properties.

It remains to distinguish and count the components obtained at each \((0,\mu_j)\). The corresponding subgroups satisfy \eqref{eq:fixed-point-separation} pairwise by Lemma~\ref{lem:intersection-block-actions}. Since every nontrivial solution in these components changes sign, Proposition~\ref{prop:geometric-separation} shows that their nontrivial parts are pairwise \(G\)-geometrically distinct and disjoint. There are \(\lfloor m/2\rfloor-1\) admissible block subgroups. In each projective case, the additional subgroup increases this number to \(\lfloor m/2\rfloor\) when \(j\) is even.
\end{proof}

\begin{Remark}
The conclusion of Lemma~\ref{lem:intersection-block-actions} remains valid if \(K_1\) is included among the block subgroups. In each of the three cases, however, \(K_1\) has fixed points, so \(d_{K_1}=0\), and the embedding of \((H^1(M))^{K_1}\) into \(L^{2n/(n-2)}(M)\) is not compact. Thus the compactness argument used in Theorem~\ref{thm:multiplicity-block-actions} does not apply to \(K_1\). For equation \eqref{eq:main} with \(2<p<2n/(n-2)\), the ordinary compact Sobolev embedding allows \(K_1\) to be included. The same argument therefore gives multiplicity bounds larger by one than those established for the critical exponent in Theorem~\ref{thm:multiplicity-block-actions}.
\end{Remark}

\begin{Remark}
The same argument also applies to equation \eqref{eq:main} with
supercritical exponents. A subgroup \(K\) considered above gives
an unbounded continuum at each positive \(K\)-invariant eigenvalue
whenever \(n-d_K\leq2\), or whenever
\(2<p<2(n-d_K)/(n-d_K-2)\) if \(n-d_K>2\).
Counting the subgroups satisfying this condition at a common
eigenvalue gives a lower bound on the number of geometrically
distinct continua bifurcating from that eigenvalue.
For the critical equation, it suffices to know that \(d_K>0\);
for supercritical exponents, the admissible range depends on
the value of \(d_K\). These values are computed in
Lemma~\ref{lem:minimal-orbit-dimensions} in the Appendix.
\end{Remark}

\begin{Remark}
Since the equation is \(G\)-invariant, one might also try to prove unboundedness using equivariant degree. Such an argument was used in \cite{Stefaniak2026}, where subcritical growth provides compactness on the full Sobolev space and allows the use of a torus action and its weights. Even in the subcritical case, however, unboundedness of a component in the full solution space does not imply unboundedness of the components obtained in prescribed fixed-point spaces. Thus an ambient unboundedness result alone would not give the multiplicity of geometrically distinct continua established here.

For critical and supercritical exponents, there is a further difficulty. Compactness is recovered by restricting to \((H^1(M))^K\), which is generally not preserved by \(G\). This space carries an action of \(N_G(K)/K\), where \(N_G(K)\) is the normalizer of \(K\) in \(G\), but this quotient may be finite; see \cite[Proposition~3.10 and Corollary~3.11]{Marzantowicz2019} for orthogonal block subgroups. The argument based on torus weights therefore does not transfer directly to these spaces. 
\end{Remark}

\section{Appendix}\label{appendix}

We give the proof of Lemma~\ref{lem:maximal-symmetric-subgroup} and compute the smallest orbit dimensions for several classical cohomogeneity one actions.

\subsection{Maximality of symmetric subgroups}\label{subsec:maxim}

\begin{proof}[Proof of Lemma~\ref{lem:maximal-symmetric-subgroup}]
Let \(\fg\), \(\fh\), and \(\fl\) denote the Lie algebras of \(G\), \(H\), and \(L\), respectively, and suppose that \(\fh\neq\fl\). We will show that \(\fl=\fg\). Since a connected Lie subgroup is determined by its Lie algebra, this will prove the lemma.

Consider the Cartan decomposition \(\fg=\fh\oplus\fp\) associated with the symmetric pair \((\fg,\fh)\). We have \([\fh,\fp]\subset\fp\); see \cite[Section~5.7]{Gurarie}. Thus \(\fp\), equipped with the action \(\rho(X)Y=[X,Y]\) for \(X\in\fh\) and \(Y\in\fp\), is a representation of \(\fh\). Since \(G\) is compact and \(\fg\) is simple, this representation is irreducible by case~I of \cite[Theorem~VIII.5.3]{Helgason}.

Consider \(\fl\cap\fp\). It is nonzero: choose \(Z\in\fl\setminus\fh\) and write \(Z=X+Y\), where \(X\in\fh\) and \(Y\in\fp\setminus\{0\}\). Since \(\fh\subset\fl\), we have \(Y=Z-X\in\fl\cap\fp\). This subspace is also invariant under \(\rho\): if \(X\in\fh\) and \(Y\in\fl\cap\fp\), then \([X,Y]\in\fl\), since \(\fh\subset\fl\), and \([X,Y]\in\fp\), since \([\fh,\fp]\subset\fp\). Irreducibility therefore gives \(\fl\cap\fp=\fp\). Hence \(\fl\) contains both \(\fh\) and \(\fp\), so \(\fl=\fg\).
\end{proof}

\subsection{Smallest orbit dimensions}\label{subsec:smallest}

We compute the quantity \(d_K\) defined in \eqref{eq:dK} for several classical cohomogeneity one actions. These values determine the compactness ranges in Theorem~\ref{thm:HV} and hence the supercritical exponents for which Theorem~\ref{thm:unbounded-continuum} applies.

\begin{Lemma}\label{lem:minimal-orbit-dimensions}
The following cohomogeneity one actions have the indicated smallest orbit dimensions.
\begin{enumerate}
\item For the action of \(K=SO(a)\times SO(m-a)\) on \(S^{m-1}\), where \(2\leq a\leq m-2\), one has \(d_K=\min\{a-1,m-a-1\}\).

\item\label{item:complex-block-orbits} For the action of \(K_a=S(U(a)\times U(m-a))\subset SU(m)\) on \(\bC P^{m-1}\), where \(m\geq2\) and \(1\leq a\leq\lfloor m/2\rfloor\), one has \(d_{K_a}=2(a-1)\).

\item\label{item:quaternionic-block-orbits} For the action of \(K_a=Sp(a)\times Sp(m-a)\subset Sp(m)\) on \(\bH P^{m-1}\), where \(m\geq2\) and \(1\leq a\leq\lfloor m/2\rfloor\), one has \(d_{K_a}=4(a-1)\).

\item For the action of \(SO(m)\subset SU(m)\) on \(\bC P^{m-1}\), where \(m\geq4\), one has \(d_{SO(m)}=m-1\).

\item For the action of \(U(m)\subset Sp(m)\) on \(\bH P^{m-1}\), where \(m\geq4\), one has \(d_{U(m)}=2m-2\).

\item For the action of \(K=SO(m-1)\) on the oriented real Grassmann manifold \(\mathrm{Gr}_k^+(\bR^m)\), induced by its action on \(\bR^m\) fixing the last coordinate, where \(2\leq k\leq m-2\), one has \(d_K=\min\{k(m-k-1),(k-1)(m-k)\}\).

\item For the diagonal action of \(K=\{(A,A):A\in SO(m)\}\) on \(S^{m-1}\times S^{m-1}\), where \(m\geq3\), one has \(d_K=m-1\).
\end{enumerate}
All orbits are infinite except when \(a=1\) in cases~\ref{item:complex-block-orbits} and~\ref{item:quaternionic-block-orbits}.
\end{Lemma}

\begin{proof}
In each case, the smallest orbit is one of the two nonprincipal orbits.

\begin{enumerate}
\item Let \(x\in S^{a-1}\subset\bR^a\) and \(y\in S^{m-a-1}\subset\bR^{m-a}\). The two nonprincipal orbits are \(K(x,0)=S^{a-1}\times\{0\}\) and \(K(0,y)=\{0\}\times S^{m-a-1}\). Their dimensions are \(a-1\) and \(m-a-1\), respectively, so \(d_K=\min\{a-1,m-a-1\}\).

\item Write a point of \(\bC P^{m-1}\) as \([x,y]\), where \((x,y)\in\bC^a\oplus\bC^{m-a}\). The two nonprincipal orbits are \(\{[x,0]:x\in\bC^a\setminus\{0\}\}\) and \(\{[0,y]:y\in\bC^{m-a}\setminus\{0\}\}\). They are diffeomorphic to \(\bC P^{a-1}\) and \(\bC P^{m-a-1}\), respectively, and hence have dimensions \(2(a-1)\) and \(2(m-a-1)\). Since \(a\leq m-a\), it follows that \(d_{K_a}=2(a-1)\).

\item The quaternionic case follows in the same way, with \(\bC\) replaced by \(\bH\). The two nonprincipal orbits are diffeomorphic to \(\bH P^{a-1}\) and \(\bH P^{m-a-1}\), and hence have dimensions \(4(a-1)\) and \(4(m-a-1)\). Since \(a\leq m-a\), it follows that \(d_{K_a}=4(a-1)\).

\item By \cite[Table~F]{GroWilZil}, the two nonprincipal orbits of the \(SO(m)\)-action on \(\bC P^{m-1}\) are \(SO(m)/(SO(2)\times SO(m-2))\) and \(SO(m)/O(m-1)\). Their dimensions are \(2m-4\) and \(m-1\), respectively. For \(m\geq4\), this gives \(d_{SO(m)}=m-1\).

\item By \cite[Table~F]{GroWilZil}, the two nonprincipal orbits of the \(U(m)\)-action on \(\bH P^{m-1}\) are \(U(m)/(SU(2)\times U(m-2))\) and \(U(m)/(U(1)\times U(m-1))\). Their dimensions are \(4m-7\) and \(2m-2\), respectively. For \(m\geq4\), this gives \(d_{U(m)}=2m-2\).

\item The group \(K=SO(m-1)\) fixes \(e_m\) and preserves \(e_m^\perp\). One nonprincipal orbit consists of the oriented \(k\)-planes contained in \(e_m^\perp\). It is diffeomorphic to \(\mathrm{Gr}_k^+(\bR^{m-1})\) and has dimension \(k(m-k-1)\).

The other nonprincipal orbit consists of the oriented \(k\)-planes containing \(e_m\). Intersecting such a plane with \(e_m^\perp\) identifies this orbit with \(\mathrm{Gr}_{k-1}^+(\bR^{m-1})\), where the orientation of the intersection is determined using \(e_m\). Its dimension is \((k-1)(m-k)\). Therefore \(d_K=\min\{k(m-k-1),(k-1)(m-k)\}\).

\item Two pairs \((x,y)\) and \((x',y')\) belong to the same diagonal orbit if and only if \(\langle x,y\rangle=\langle x',y'\rangle\). Thus the map \(f\colon(S^{m-1}\times S^{m-1})/K\to[-1,1]\), defined by \(f(K(x,y))=\langle x,y\rangle\), identifies the orbit space with \([-1,1]\). The nonprincipal orbits corresponding to the endpoints are \(\{(x,x):x\in S^{m-1}\}\) and \(\{(x,-x):x\in S^{m-1}\}\). Both are diffeomorphic to \(S^{m-1}\), so \(d_K=m-1\).
\end{enumerate}

Since the acting groups are compact, all their orbits are infinite if and only if \(d_K>0\). The remaining assertions follow from the formulas above.
\end{proof}

\begin{Remark}
For the diagonal action on \(S^{m-1}\times S^{m-1}\), the restriction \(m\geq3\) ensures that the transitive group \(SO(m)\times SO(m)\) is semisimple, so Lemma~\ref{lem:orbit-spaces-interval} gives interval orbit space. For \(m=2\), the formula \(d_K=m-1\) still holds. However, \(SO(2)\times SO(2)\) is abelian and hence not semisimple, so the lemma does not apply; in this case, the orbit space is a circle.
\end{Remark}

\textbf{AI disclosure.} During the research and preparation of this article, the author used ChatGPT to assist in navigating the literature, exploring possible proof strategies, and improving the exposition. The mathematical arguments and references in the final manuscript were verified by the author, who takes full responsibility for the content.


\begin{thebibliography}{99}
\bibitem{AlexandrinoBettiol2015} M.~M. Alexandrino and R.~G. Bettiol, \emph{Lie Groups and Geometric Aspects of Isometric Actions}, Springer, Cham, 2015.

\bibitem{Aubin1982} T.~Aubin, \emph{Nonlinear Analysis on Manifolds. Monge--Amp\`ere Equations}, Grundlehren der mathematischen Wissenschaften 252, Springer-Verlag, New York, 1982.

\bibitem{BetJulPet2021} A.~Betancourt de la Parra, J.~Julio-Batalla and J.~Petean, \emph{Global bifurcation techniques for Yamabe type equations on Riemannian manifolds}, Nonlinear Anal. 202 (2021), 112140.

\bibitem{BettiolPiccione2013} R.~G. Bettiol and P.~Piccione, \emph{Bifurcation and local rigidity of homogeneous solutions to the Yamabe problem on spheres}, Calc. Var. Partial Differential Equations 47 (2013), 789--807.

\bibitem{Bredon} G.~E. Bredon, \emph{Introduction to Compact Transformation Groups}, Pure and Applied Mathematics 46, Academic Press, New York--London, 1972.

\bibitem{ClappFernandez2017} M.~Clapp and J.~C. Fern\'andez, \emph{Multiplicity of nodal solutions to the Yamabe problem}, Calc. Var. Partial Differential Equations 56(5) (2017), Paper No.~145, 22 pp.

\bibitem{ClappSaldanaSzulkin2021} M.~Clapp, A.~Salda\~na and A.~Szulkin, \emph{Phase separation, optimal partitions, and nodal solutions to the Yamabe equation on the sphere}, Int. Math. Res. Not. IMRN 2021 (2021), no.~5, 3633--3652.

\bibitem{FerPalPet2020} J.~C. Fern\'andez, O.~Palmas and J.~Petean, \emph{Supercritical elliptic problems on the round sphere and nodal solutions to the Yamabe problem in projective spaces}, Discrete Contin. Dyn. Syst. 40(4) (2020), 2495--2514.

\bibitem{GroWilZil} K.~Grove, B.~Wilking and W.~Ziller, \emph{Positively curved cohomogeneity one manifolds and \(3\)-Sasakian geometry}, J. Differential Geom. 78(1) (2008), 33--111.

\bibitem{Gurarie} D.~Gurarie, \emph{Symmetries and Laplacians: Introduction to Harmonic Analysis, Group Representations and Applications}, North-Holland Mathematics Studies 174, North-Holland, Amsterdam, 1992.

\bibitem{HebVau} E.~Hebey and M.~Vaugon, \emph{Sobolev spaces in the presence of symmetries}, J. Math. Pures Appl. 76(10) (1997), 859--881.

\bibitem{HealeyKielhofer1993} T.~J. Healey and H.~Kielh\"ofer, \emph{Separation of global solution branches of elliptic systems with symmetry via nodal properties}, Nonlinear Anal. 21(9) (1993), 665--684.

\bibitem{Helgason} S.~Helgason, \emph{Differential Geometry, Lie Groups, and Symmetric Spaces}, Pure and Applied Mathematics 80, Academic Press, New York, 1978.

\bibitem{Henry2019} G.~Henry, \emph{Isoparametric functions and nodal solutions of the Yamabe equation}, Ann. Global Anal. Geom. 56(2) (2019), 203--219.

\bibitem{KnappLie} A.~W. Knapp, \emph{Lie Groups Beyond an Introduction}, digital 2nd ed., Progress in Mathematics 140, published by the author, East Setauket, NY, 2023.

\bibitem{Kollross} A.~Kollross, \emph{A classification of hyperpolar and cohomogeneity one actions}, Trans. Amer. Math. Soc. 354(2) (2002), 571--612.

\bibitem{LeeParker1987} J.~M. Lee and T.~H. Parker, \emph{The Yamabe problem}, Bull. Amer. Math. Soc. (N.S.) 17(1) (1987), 37--91.

\bibitem{LimaPiccioneZedda2012} L.~L. de Lima, P.~Piccione and M.~Zedda, \emph{On bifurcation of solutions of the Yamabe problem in product manifolds}, Ann. Inst. H. Poincar\'e C Anal. Non Lin\'eaire 29 (2012), no.~2, 261--277.

\bibitem{Marzantowicz2019} W.~Marzantowicz, \emph{Geometrically distinct solutions given by symmetries of variational problems with the \(O(N)\)-symmetry}, Commun. Contemp. Math. 21(8) (2019), 1850080.

\bibitem{PetBar2019} J.~Petean and H.~Barrantes González, \emph{A note on solutions of Yamabe-type equations on products of spheres}, Proc. Amer. Math. Soc. 147(7) (2019), 3143--3153.

\bibitem{Rabinowitz1971} P.~H. Rabinowitz, \emph{Some global results for nonlinear eigenvalue problems}, J. Funct. Anal. 7(3) (1971), 487--513.

\bibitem{SmoWass1990} J.~Smoller and A.~G. Wasserman, \emph{Bifurcation and symmetry-breaking}, Invent. Math. 100(1) (1990), 63--95.

\bibitem{Stefaniak2026} P.~Stefaniak, \emph{Unbounded sets of solutions of non-cooperative elliptic systems on symmetric spaces}, J. Differential Equations 482 (2026), 114741.

\end{thebibliography}
\end{document}